\documentclass{amsart}
\usepackage{amsmath}
\usepackage{amssymb}
\usepackage{amsthm}
\usepackage{epsfig}
\usepackage{amsfonts}
\usepackage{amscd}
\usepackage{mathrsfs}
\usepackage{enumerate}
\usepackage{latexsym}

\newtheorem{theorem}{Theorem}[section]
\newtheorem{lemma}[theorem]{Lemma}
\newtheorem{proposition}[theorem]{Proposition}

\theoremstyle{definition}
\newtheorem{definition}[theorem]{Definition}

\theoremstyle{remark}
\newtheorem{remark}[theorem]{Remark}

\begin{document}

\title[Continuous bounded functions with separable support]{The Banach algebra of continuous bounded functions with separable support}

\author{M. R. Koushesh}
\address{Department of Mathematical Sciences, Isfahan University of Technology, Isfahan 84156--83111, Iran}
\address{School of Mathematics, Institute for Research in Fundamental Sciences (IPM), P.O. Box: 19395--5746, Tehran, Iran}
\email{koushesh@cc.iut.ac.ir}
\thanks{This research was in part supported by a grant from IPM (No. 90030052).}


%\date{October 17, 2011 and, in revised form, .}


\begin{abstract}
We prove a commutative Gelfand--Naimark type theorem, by showing that the set $C_s(X)$ of continuous bounded (real or complex valued) functions with separable support on a locally separable metrizable space $X$ (provided with the supremum norm) is a Banach algebra, isometrically isomorphic to  $C_0(Y)$, for some unique (up to homeomorphism) locally compact Hausdorff space $Y$. The space $Y$, which we explicitly construct as a subspace of the Stone--\v{C}ech compactification of $X$, is countably compact, and if $X$ is non--separable, is moreover non--normal; in addition $C_0(Y)=C_{00}(Y)$. When the underlying field of scalars is the complex numbers, the space $Y$ coincides with the spectrum of the $\mbox{C}^*$--algebra $C_s(X)$. Further, we find the dimension of the algebra $C_s(X)$.
\end{abstract}

\maketitle

\section{Introduction}

Throughout this article the underlying field of scalars (which is fixed throughout each discussion) is assumed to be either the real field $\mathbb{R}$ or the complex field $\mathbb{C}$, unless specifically stated otherwise. All spaces considered are assumed to be Hausdorff. Let $X$ be a completely regular space. Denote by $C_b(X)$ the set of all continuous bounded functions on $X$. If $f\in C_b(X)$, the {\em zero--set} of $f$, denoted by $\mbox{Z}(f)$, is $f^{-1}(0)$, the {\em cozero--set} of $f$, denoted by $\mbox{Coz}(f)$, is $X\backslash \mbox{Z}(f)$, and the {\em support} of $f$, denoted by $\mbox{supp}(f)$, is $\mbox{cl}_X \mbox{Coz}(f)$. Let \[\mbox{Coz}(X)=\big\{\mbox{Coz}(f):f\in C_b(X)\big\}.\]
The elements of $\mbox{Coz}(X)$ are called {\em cozero--sets} of $X$. Denote by $C_0(X)$ the set of all $f\in C_b(X)$ which vanish at infinity (i.e., $|f|^{-1}([\epsilon,\infty))$ is compact for each $\epsilon>0$) and denote by $C_{00}(X)$ the set of all $f\in C_b(X)$ with compact support.

The purpose of this article is to show that the set $C_s(X)$ of continuous bounded (real or complex valued) functions with separable support on a locally separable metrizable space $X$ (provided with the supremum norm) is a Banach algebra, which is isometrically isomorphic to the Banach algebra $C_0(Y)$, for some unique (up to homeomorphism) locally compact space $Y$. (This is a direct conclusion of the commutative Gelfand--Naimark Theorem when the underlying field of scalars is $\mathbb{C}$, provided that one assumes $C_s(X)$ to be a Banach algebra.) The space $Y$, which  is explicitly constructed as a subspace of the Stone--\v{C}ech compactification of $X$, is shown to be countably compact and non--normal (for the latter, provided that $X$ is non--separable) and moreover $C_0(Y)=C_{00}(Y)$. In the case when the underlying field of scalars is $\mathbb{C}$, the space $Y$ coincides with the spectrum of the $\mbox{C}^*$--algebra $C_s(X)$. Further, the dimension of the algebra $C_s(X)$ is found in terms of the density of $X$.

We now review briefly some known facts from the theories of the Stone--\v{C}ech compactification and metrizable spaces. Additional information on these subjects may be found in \cite{E} and \cite{GJ}.

\subsection*{1.1. The Stone--\v{C}ech compactification.} Let $X$ be a completely regular space. The {\em Stone--\v{C}ech compactification} $\beta X$ of $X$ is the compactification of $X$ which is characterized among all compactifications of $X$ by the following property: Every continuous $f:X\rightarrow K$, where $K$ is a compact space, is continuously extendable over $\beta X$; denote by $f_\beta$ this continuous extension of $f$. Use will be made in what follows of the following properties of $\beta X$. (See Sections 3.5 and 3.6 of \cite{E}.)
\begin{itemize}
  \item $X$ is locally compact if and only if $X$ is open in $\beta X$.
  \item Any open--closed subspace of $X$ has open--closed closure in $\beta X$.
  \item If $X\subseteq T\subseteq\beta X$ then $\beta T=\beta X$.
  \item If $X$ is normal then $\beta T=\mbox{cl}_{\beta X}T$ for any closed subspace $T$ of $X$.
\end{itemize}

\subsection*{1.2. Separability and local separability in metrizable spaces.} The {\em density} of a space $X$, denoted by $d(X)$, is the smallest cardinal number of the form $|D|$, where $D$ is dense in $X$. Therefore, a space $X$ is separable if $d(X)\leq\aleph_0$. Note that in any metrizable space the three notions of separability, being Lindel\"{o}f, and second countability coincide; thus any subspace of a separable metrizable space is separable.
A space is called {\em locally separable} if each of its points has a separable open neighborhood. By a theorem of Alexandroff, any locally separable metrizable space $X$ can be represented as a disjoint union
\[X=\bigcup_{i\in I}X_i,\]
where $I$ is an index set, and $X_i$ is a non--empty separable open--closed subspace of $X$ for each $i\in I$. (See Problem 4.4.F of \cite{E}.) Note that $d(X)=|I|$, provided that $I$ is an infinite set.

\section{The Banach algebra $C_s(X)$ of continuous bounded functions with separable support on a locally separable metrizable space $X$}

\begin{definition}
For any metrizable space $X$ let
\[C_s(X)=\big\{f\in C_b(X):\mbox{supp}(f)\mbox{ is separable}\big\}.\]
\end{definition}

Recall that any subspace of a separable metrizable space is separable. Also, note that any metrizable space $X$ is completely regular, that is, if $x\in X$ and $U$ is an  open neighborhood of $x$ in $X$, then there exists a continuous $f:X\rightarrow [0,1]$ such that $f(x)=1$ and $f|(X\backslash U)\equiv 0$.

\begin{proposition}\label{YHP}
Let $X$ be a metrizable space. Then $C_s(X)$ is a closed subalgebra of $C_b(X)$. Furthermore, if $X$ is locally separable, then $C_s(X)$ is unital if and only if $X$ is separable.
\end{proposition}

\begin{proof}
To show that $C_s(X)$ is a subalgebra of $C_b(X)$, let $f,g\in C_s(X)$. Note that
\[\mbox{Coz}(f+g)\subseteq \mbox{Coz}(f)\cup \mbox{Coz}(g),\]
and $\mbox{Coz}(f)\cup \mbox{Coz}(g)$ is separable, as it is contained in $\mbox{supp}(f)\cup\mbox{supp}(g)$ and the latter is so. Thus $\mbox{Coz}(f+g)$ is separable, and then so is its closure $\mbox{supp}(f+g)$ in $X$. That is $f+g\in C_s(X)$. Similarly, $fg\in C_s(X)$.

To show that $C_s(X)$ is closed in $C_b(X)$, let $f_1,f_2,\ldots $ be a sequence in $C_s(X)$ converging to some $f\in C_b(X)$. Note that
\[\mbox{Coz}(f)\subseteq\bigcup_{n=1}^\infty \mbox{Coz}(f_n)=C\]
and $\mbox{Coz}(f_n)$ is separable for each $n$, as it is contained in $\mbox{supp}(f_n)$. Thus the countable union $C$ is also separable. But then the subspace $\mbox{Coz}(f)$ of $C$ is separable, and thus so is its closure $\mbox{supp}(f)$ in $X$. Therefore $f\in C_s(X)$.

It is obvious that if $X$ is separable, then $C_s(X)$ is unital with the unit element $\mathbf{1}$ (the function which maps every element of $X$ to $1$). Now suppose that $X$ is locally separable and that $C_s(X)$ is unital. We show that $X$ is separable. Let $u$ be the unit element of $C_s(X)$. Let $x\in X$. Since $X$ is locally separable, there exist a separable open neighborhood $U_x$ of $x$ in $X$. Let $f_x:X\rightarrow [0,1]$ be continuous with $f_x(x)=1$ and $f_x|(X\backslash U_x)\equiv 0$. Since $\mbox{Coz}(f_x)$ is separable (as $\mbox{Coz}(f_x)\subseteq U_x$), so is its closure $\mbox{supp}(f_x)$ in $X$, and thus $f_x\in C_s(X)$. But then $u(x).f_x(x)=f_x(x)$ implies that $u(x)=1$. Therefore $u=\mathbf{1}$, and thus $X=\mbox{supp}(u)$ is separable.
\end{proof}

The following subspace of $\beta X$ will play a crucial role in our study.

\begin{definition}\label{HGA}
For  any metrizable space $X$ let
\[\lambda X=\bigcup\big\{\mbox{int}_{\beta X} \mbox{cl}_{\beta X}C:C\in \mbox{Coz}(X)\mbox{ is separable}\big\}.\]
\end{definition}

Observe that $\lambda X$ coincides with $\lambda_\mathcal{P} X$, as defined in \cite{Ko3} (also, in \cite{Ko4} and \cite{Ko5}), with $\mathcal{P}$ taken to be separability (provided that $X$ is metrizable).

Note that if $X$ is a space and $D$ is a dense subspace of $X$, then $\mbox{cl}_XU=\mbox{cl}_X(U\cap D)$ for every open subspace $U$ of $X$; this simple observation will be used below.

\begin{lemma}\label{JHG}
Let $X$ be a metrizable space. Then $X$ is locally separable if and only if $X\subseteq\lambda X$.
\end{lemma}

\begin{proof}
Suppose that $X$ is locally separable. Let $x\in X$, and let $U$ be a separable open neighborhood of $x$ in $X$. Let $f:X\rightarrow[0,1]$ be continuous with $f(x)=0$ and $f|(X\backslash U)\equiv 1$. Let $C=f^{-1}([0,1/2))$, then $C\in \mbox{Coz}(X)$. (To see the latter, define
\[g=\max\big\{0,1/2-f\big\}\]
and then observe that $C=\mbox{Coz}(g)$.) Note that $C$ is separable, as $C\subseteq U$. Since
\begin{eqnarray*}
f_\beta^{-1}\big([0,1/2)\big)&\subseteq&\mbox{cl}_{\beta X}f^{-1}_\beta\big([0,1/2)\big)\\&=&\mbox{cl}_{\beta X}\big(X\cap f_\beta^{-1}\big([0,1/2)\big)\big)=\mbox{cl}_{\beta X}f^{-1}\big([0,1/2)\big)=\mbox{cl}_{\beta X}C
\end{eqnarray*}
it follows that
\[x\in f_\beta^{-1}\big([0,1/2)\big)\subseteq\mbox{int}_{\beta X}\mbox{cl}_{\beta X}C\subseteq\lambda X.\]

Now, suppose that $X\subseteq\lambda X$. Let $x\in X$. Then $x\in \lambda X$, which implies that $x\in \mbox{int}_{\beta X}\mbox{cl}_{\beta X}D$ for some separable $D\in \mbox{Coz}(X)$. Let
\[V=X\cap\mbox{int}_{\beta X}\mbox{cl}_{\beta X}D.\]
Then $V$ is an open neighborhood of $x$ in $X$, and it is separable, as
\[V\subseteq X\cap\mbox{cl}_{\beta X}D=\mbox{cl}_X D,\]
and the latter is so.
\end{proof}

\begin{definition}\label{WWA}
Let $X$ be a locally separable metrizable space. For any $f\in C_b(X)$ denote $f_\lambda=f_\beta|\lambda X$.
\end{definition}

Note that by Lemma \ref{JHG} the function $f_\lambda$ extends $f$.

\begin{lemma}\label{TES}
Let $X$ be a locally separable metrizable space. For any $f\in C_b(X)$ the following are equivalent:
\begin{itemize}
\item[\rm(1)] $f\in C_s(X)$.
\item[\rm(2)] $f_\lambda\in C_0(\lambda X)$.
\end{itemize}
\end{lemma}

\begin{proof}
(1) {\em  implies} (2). Note that $\mbox{Coz}(f)$, being a subspace of $\mbox{supp}(f)$, is separable. Now
\[\mbox{Coz}(f_\beta)\subseteq\mbox{cl}_{\beta X}\mbox{Coz}(f_\beta)=\mbox{cl}_{\beta X}\big(X\cap \mbox{Coz}(f_\beta)\big)=\mbox{cl}_{\beta X}\mbox{Coz}(f)\]
and thus
\[\mbox{Coz}(f_\beta)\subseteq\mbox{int}_{\beta X}\mbox{cl}_{\beta X}\mbox{Coz}(f)\subseteq\lambda X.\]
For any $\epsilon>0$, the space
\[|f_\lambda|^{-1}\big([\epsilon,\infty)\big)=|f_\beta|^{-1}\big([\epsilon,\infty)\big),\]
being closed in $\beta X$, is compact.

(2) {\em  implies} (1). Let $n$ be a positive integer. Since $|f_\lambda|^{-1}([1/n,\infty))$ is a compact subspace of $\lambda X$, we have
\begin{equation}\label{JB}
|f_\lambda|^{-1}\big([1/n,\infty)\big)\subseteq\mbox{int}_{\beta X}\mbox{cl}_{\beta X}C_1\cup\cdots\cup\mbox{int}_{\beta X}\mbox{cl}_{\beta X}C_k
\end{equation}
for some separable $C_1,\ldots,C_k\in \mbox{Coz}(X)$. Intersecting both sides of (\ref{JB}) with $X$, it follows that $|f|^{-1}([1/n,\infty))$, being a subspace of the separable metrizable space
\[\mbox{cl}_XC_1\cup\cdots\cup\mbox{cl}_XC_k,\]
is separable. But then
\[\mbox{Coz}(f)=\bigcup_{n=1}^\infty |f|^{-1}\big([1/n,\infty)\big)\]
is also separable, and thus so is its closure $\mbox{supp}(f)$ in $X$.
\end{proof}

Observe that any open--closed subspace $A$ of a space $X$ is a cozero--set of $X$; indeed, $A=\mbox{Coz}(f)$, where $f=\chi_A$ is the characteristic function of $A$.

\begin{lemma}\label{PDS}
Let $X$ be a locally separable metrizable space. Let $X$ be represented as a disjoint union $X=\bigcup_{i\in I}X_i$, such that $X_i$ is a separable open--closed subspace of $X$ for each $i\in I$. Then
\[\lambda X=\bigcup\Big\{\mbox{\em cl}_{\beta X}\Big(\bigcup_{i\in J}X_i\Big):J\subseteq I\mbox{ is countable}\Big\}.\]
\end{lemma}

\begin{proof}
Denote
\[\mu X=\bigcup\Big\{\mbox{cl}_{\beta X}\Big(\bigcup_{i\in J}X_i\Big):J\subseteq I\mbox{ is countable}\Big\}.\]

To show that $\lambda X\subseteq\mu X$, let $C\in \mbox{Coz}(X)$ be separable. Then $C$ is Lindel\"{o}f and therefore $C\subseteq\bigcup_{i\in J}X_i$ for some countable $J\subseteq I$. Thus
\[\mbox{cl}_{\beta X}C\subseteq\mbox{cl}_{\beta X}\Big(\bigcup_{i\in J}X_i\Big).\]

Next, we show that $\mu X\subseteq\lambda X$. Let $J\subseteq I$ be countable. Then $D=\bigcup_{i\in J}X_i$ is a cozero--set of $X$, as it is open--closed in $X$, and it is separable. Also, since $D$ is open--closed in $X$, its closure $\mbox{cl}_{\beta X}D$ in $\beta X$ is open--closed in $\beta X$. Thus
\[\mbox{cl}_{\beta X}D=\mbox{int}_{\beta X} \mbox{cl}_{\beta X}D\subseteq\lambda X.\]
\end{proof}

Let $X$ be a locally compact non--compact space. It is known that $C_0(X)=C_{00}(X)$ if and only if every $\sigma$--compact subspace of $X$ is contained in a compact subspace of $X$. (See Problem 7G.2 of \cite{GJ}.) In particular, $C_0(X)=C_{00}(X)$ implies that $X$ is countably compact, and thus non--paracompact, as every countably compact paracompact space is compact. (See Theorem 5.1.20 of \cite{E}.) Below, it will be shown that $C_0(\lambda X)=C_{00}(\lambda X)$ for any locally separable metrizable space $X$.

\begin{lemma}\label{HGFS}
Let $X$ be a locally separable metrizable space. Then $\lambda X$ is locally compact and $C_0(\lambda X)=C_{00}(\lambda X)$. In particular, $\lambda X$ is countably compact.
\end{lemma}

\begin{proof}
Note that $\lambda X$, being open in $\beta X$, is locally compact. To prove the lemma, it suffices to show that every $\sigma$--compact subspace of $\lambda X$ is contained in a compact subspace of $\lambda X$. Let $T$ be a $\sigma$--compact subspace of $\lambda X$. Then
\[T=\bigcup_{n=1}^\infty T_n,\]
where $T_n$ is compact for each positive integer $n$. Assume the representation of $X$ given in Part 1.2. Using Lemma \ref{PDS}, for each positive integer $n$ (by compactness of $T_n$ and the fact that $\mbox{cl}_{\beta X}(\bigcup_{i\in J}X_i)$ is open in $\beta X$, as $\bigcup_{i\in J}X_i$ is open--closed in $X$ for each countable $J\subseteq I$) there are countable $J_1^n,\ldots,J_{k_n}^n\subseteq I$ with
\[T_n\subseteq\mbox{cl}_{\beta X}\Big(\bigcup_{i\in J_1^n}X_i\Big)\cup\cdots\cup\mbox{cl}_{\beta X}\Big(\bigcup_{i\in J_{k_n}^n}X_i\Big).\]
If we now let
\[J=\bigcup_{n=1}^\infty(J_1^n\cup\cdots\cup J_{k_n}^n),\]
then $J$ is countable, and $\mbox{cl}_{\beta X}(\bigcup_{i\in J}X_i)$ is a compact subspace of $\lambda X$ containing $T$.
\end{proof}

Let $D$ be an uncountable discrete space. Let $E$ be the subspace of $\beta D\backslash D$ consisting of elements in the closure in $D$ of countable subsets of $D$. Then $E=\lambda D\backslash D$. (Observe that separable cozero--sets of $D$ are exactly countable subspace of $D$, and each subspace of $D$, being open--closed in $D$, has open closure in $\beta D$.) In \cite{W}, the author proves the existence of a continuous (2--valued) function $f:E\rightarrow[0,1]$ which is not continuously extendible over $\beta D\backslash D$. This, in particular, proves that $\lambda D$ is not normal. (To see this, suppose, to the contrary, that $\lambda D$ is normal. Note that $E$ is closed in $\lambda D$, as $D$, being locally compact, is open in $\beta D$. By the Tietze--Urysohn Extension Theorem, $f$ is extendible to a continuous bounded function over $\lambda D$, and thus over $\beta(\lambda D)=\beta D$. But this is not possible.) This fact will be used below to show that in general $\lambda X$ is non--normal for any locally separable non--separable metrizable space $X$. This, together with Lemma \ref{HGFS}, provides an example of a locally compact countably compact non--normal space $Y$ with $C_0(Y)=C_{00}(Y)$.

Observe that if $X$ is a space and $D\subseteq X$, then
\[U\cap\mbox{cl}_XD=\mbox{cl}_X(U\cap D)\]
for every open--closed subspace $U$ of $X$; this simple observation will be used below.

\begin{lemma}\label{OPS}
Let $X$ be a locally separable non--separable metrizable space. Then $\lambda X$ is non--normal.
\end{lemma}

\begin{proof}
Assume the representation of $X$ given in Part 1.2. Choose some $x_i\in X_i$ for each $i\in I$. Consider the subspace
\[D=\{x_i:i\in I\}\]
of $X$. Then $D$ is a closed discrete subspace of $X$, and since $X$ is non--separable, it is uncountable. Suppose to the contrary that $\lambda X$ is normal. Using Lemma \ref{PDS}, the space
\[\lambda X\cap\mbox{cl}_{\beta X}D=\bigcup\Big\{\mbox{cl}_{\beta X}\Big(\bigcup_{i\in J}X_i\Big)\cap\mbox{cl}_{\beta X}D:J\subseteq I\mbox{ is countable}\Big\},\]
being closed in $\lambda X$, is normal. Now, let $J\subseteq I$ be countable. Since $\mbox{cl}_{\beta X}(\bigcup_{i\in J}X_i)$ is open in $\beta X$ (as $\bigcup_{i\in J}X_i$ is open--closed in $X$) we have
\begin{eqnarray*}
\mbox{cl}_{\beta X}\Big(\bigcup_{i\in J}X_i\Big)\cap\mbox{cl}_{\beta X}D&=&\mbox{cl}_{\beta X}\Big(\mbox{cl}_{\beta X}\Big(\bigcup_{i\in J}X_i\Big)\cap D\Big)\\&=&\mbox{cl}_{\beta X}\Big(\bigcup_{i\in J}X_i\cap D\Big)=\mbox{cl}_{\beta X}\big(\{x_i:i\in J\}\big).
\end{eqnarray*}
But $\mbox{cl}_{\beta X}D=\beta D$, as $D$ is closed in (the normal space) $X$. Therefore
\[\mbox{cl}_{\beta X}\big(\{x_i:i\in J\}\big)=\mbox{cl}_{\beta X}\big(\{x_i:i\in J\}\big)\cap\mbox{cl}_{\beta X}D=\mbox{cl}_{\beta D}\big(\{x_i:i\in J\}\big).\]
Thus
\[\lambda X\cap\mbox{cl}_{\beta X}D=\lambda D,\]
contradicting the fact that $\lambda D$ is not normal.
\end{proof}

A version of the classical Banach--Stone Theorem states that if $X$ and $Y$ are locally compact spaces, the  Banach algebras $C_0(X)$ and $C_0(Y)$ are isometrically isomorphic if and only if the spaces $X$ and $Y$ are homeomorphic (see Theorem 7.1 of \cite{Be}); this will be used in the proof of the following main  theorem.

\begin{theorem}\label{TRES}
Let $X$ be a locally separable metrizable space. Then $C_s(X)$ is a Banach algebra isometrically isomorphic to the Banach algebra $C_0(Y)$ for some unique (up to homeomorphism) locally compact space $Y$. The space $Y$ is countably compact, and if $X$ is non--separable, is non--normal. Furthermore, $C_0(Y)=C_{00}(Y)$.
\end{theorem}

\begin{proof}
Let $Y=\lambda X$ and define $\psi:C_s(X)\rightarrow C_0(Y)$ by $\psi(f)=f_\lambda$ for any $f\in C_s(X)$. By Lemma \ref{TES} the function $\psi$ is well--defined. It is clear that $\psi$ is an isometric homomorphism and $\psi$ is injective. (Note that $X\subseteq Y$ by Lemma \ref{JHG}.)  Let $g\in C_0(Y)$. Then $(g|X)_\lambda=g$ and thus  $g|X\in C_s(X)$ by Lemma \ref{TES}. Now $\psi(g|X)=g$. This shows that $\psi$ is surjective.
Note that by Lemma \ref{HGFS} the space $Y$ is locally compact. The uniqueness of $Y$ follows, as for any locally compact space $T$ the Banach algebra $C_0(T)$ determines the topology of $T$. Lemmas \ref{HGFS} and \ref{OPS} now complete the proof.
\end{proof}

\begin{remark}
Theorem \ref{TRES} holds true if one replaces ``locally separable'' and  ``the Banach algebra of continuous bounded functions with separable support'', respectively, by ``locally Lindel\"{o}f (locally second countable, respectively)'' and  ``the Banach algebra of continuous bounded functions with Lindel\"{o}f (second countable, respectively) support''.
\end{remark}

\begin{remark}
By a version of the Banach--Stone Theorem, if $X$ and $Y$ are locally compact spaces, the  rings $C_0(X)$ and $C_0(Y)$ are isomorphic if and only if the spaces $X$ and $Y$ are homeomorphic. (See \cite{A}.) Thus, Theorem \ref{TRES} (and its subsequent results) holds true if  one replaces ``Banach algebra'' by  ``ring''.
\end{remark}

\section{The dimension of $C_s(X)$}

The Tarski Theorem states that for any infinite set $I$, there is a collection $\mathscr{A}$ of cardinality $|I|^{\aleph_0}$ consisting of countable infinite subsets of $I$, such that the intersection of any two distinct elements of $\mathscr{A}$ is finite (see \cite{Ho}); this will be used in the following.

Note that the collection of all subsets of cardinality at most $\mathfrak{m}$ in a set of cardinality $\mathfrak{n}\geq\mathfrak{m}$ has cardinality at most $\mathfrak{n}^\mathfrak{m}$.

\begin{theorem}\label{RTS}
Let $X$ be a locally separable non--separable metrizable space. Then
\[\dim C_s(X)=d(X)^{\aleph_0}.\]
\end{theorem}

\begin{proof}
Assume the representation of $X$ given in Part 1.2. Note that $I$ is infinite, as $X$ is non--separable, and $d(X)=|I|$.

Let $\mathscr{A}$ be a collection of cardinality $|I|^{\aleph_0}$ consisting of countable infinite subsets of $I$, such that the intersection of any two distinct elements of $\mathscr{A}$ is finite. Define
\[f_A=\chi_{(\bigcup_{i\in A}X_i)}\]
for any $A\in\mathscr{A}$. Then, no element of
\[\mathscr{F}=\{f_A:A\in\mathscr{A}\}\]
is a linear combination of other elements (since each element of $\mathscr{A}$ is infinite and each pair of distinct elements of $\mathscr{A}$ has finite intersection). Observe that $\mathscr{F}$ is of cardinality $|\mathscr{A}|$. This shows that
\[\dim C_s(X)\geq|\mathscr{A}|=|I|^{\aleph_0}=d(X)^{\aleph_0}.\]

To simplify the notation, denote
\[H_J=\bigcup_{i\in J}X_i\]
for any $J\subseteq I$. If $f\in C_s(X)$, then $\mbox{supp}(f)$ (being separable) is Lindel\"{o}f, and thus $\mbox{supp}(f)\subseteq H_J$, where $J\subseteq I$ is countable; therefore, it may be assumed that $f\in C_b(H_J)$. Conversely, if $J\subseteq I$ is countable, then each element of $C_b(H_J)$ can be extended trivially to an element of $C_s(X)$ (by defining it to be identically $0$ elsewhere). Thus $C_s(X)$ may be viewed as the union of all $C_b(H_J)$, where $J$ runs over all countable subsets of $I$. Note that if $J\subseteq I$ is countable, then $H_J$ is separable; thus any element of $C_b(H_J)$ is determined by its value on a countable set. This implies that for each countable $J\subseteq I$, the set $C_b(H_J)$ is of cardinality at most $2^{\aleph_0}$. Note that there are at most $|I|^{\aleph_0}$ countable $J\subseteq I$. Now
\begin{eqnarray*}
\dim C_s(X)\leq\big|C_s(X)\big|&\leq& \Big|\bigcup\big\{C_b(H_J):J\subseteq I\mbox{ is countable}\big\}\Big|\\&\leq& 2^{\aleph_0}\cdot|I|^{\aleph_0}=|I|^{\aleph_0}=d(X)^{\aleph_0},
\end{eqnarray*}
which together with the first part proves the theorem.
\end{proof}

\section{The spectrum of the $\mbox{C}^*$--algebra $C_s(X)$}

In this section the underlying field of scalars is $\mathbb{C}$. Let $A$ be a commutative Banach algebra. A non--zero algebra homomorphism $\phi:A\rightarrow\mathbb{C}$ is called a {\em character} of $A$; the set of all characters of $A$ is denoted by $\Phi_A$. If $A$ is a $\mbox{C}^*$--algebra then every algebra homomorphism $\phi:A\rightarrow\mathbb{C}$ is a $*$--homomorphism, and thus $\Phi_A$ coincides with the spectrum of $A$. Every character on $A$ is continuous, and therefore $\Phi_A$ is a subset of the space $A^*$ of continuous linear functionals on $A$; moreover, when equipped with the relative weak$^*$ topology, $\Phi_A$ turns out to be locally compact. The space $\Phi_A$ is compact (in the topology just defined) if and only if the algebra $A$ has an identity element. Given $a\in A$, let the function $\hat{a}:\Phi_A\rightarrow\mathbb{C}$ be defined by $\hat{a}(\phi)=\phi(a)$ for any $\phi\in\Phi_A$. The map $a\mapsto\hat{a}$ defines a norm--decreasing, unit--preserving algebra homomorphism from $A$ to $C_0(\Phi_A)$. This homomorphism is called the {\em Gelfand representation} of $A$. In general the representation is neither injective nor surjective. The commutative Gelfand--Naimark Theorem states that if $A$ is a commutative $\mbox{C}^*$--algebra then the Gelfand map is an isometric $*$--isomorphism.

\begin{theorem}\label{SUS}
Let $X$ be a locally separable non--separable metrizable space. Then the spectrum of $C_s(X)$ is homeomorphic to $\lambda X$, and thus is locally compact, countably compact and non--normal.
\end{theorem}

\begin{proof}
By the commutative Gelfand--Naimark Theorem the $\mbox{C}^*$--algebra $C_s(X)$ is isometrically $*$--isomorphic to $C_0(S)$, where $S$ is the spectrum of $C_s(X)$. On the other hand, $C_s(X)$ is isometrically isomorphic (as a Banach algebra) to $C_0(\lambda X)$, by (the proof of) Theorem \ref{TRES}. This implies that $C_0(S)$ is isometrically isomorphic to $C_0(\lambda X)$, which by the Banach--Stone Theorem (and Lemmas \ref{HGFS} and \ref{OPS}) gives the result.
\end{proof}

\section{The Banach algebra $C_\sigma(X)$ of continuous bounded functions with $\sigma$--compact support}

Recall that in any locally compact space, $\sigma$--compactness coincides with being Lindel\"{o}f (see Problem 3.8.C of \cite{E}); thus, in any locally compact metrizable space, $\sigma$--compactness and separability coincide.

The following variation of Theorem \ref{TRES} might be of some interest; results dual to Theorems \ref{RTS} and \ref{SUS} may be stated and proved analogously.

\begin{definition}
For any metrizable  space $X$ let
\[C_\sigma(X)=\big\{f\in C_b(X):\mbox{supp}(f)\mbox{ is $\sigma$--compact}\big\}.\]
\end{definition}

\begin{theorem}\label{YES}
Let $X$ be a locally compact metrizable space. Then $C_\sigma(X)$ is a Banach algebra isometrically isomorphic to the Banach algebra $C_0(Y)$ for some unique (up to homeomorphism) locally compact space $Y$. The space $Y$ is countably compact, and if $X$ is non--$\sigma$--compact, is non--normal. Furthermore, $C_0(Y)=C_{00}(Y)$.
\end{theorem}

\begin{proof}
This follows from Theorem \ref{TRES} and the fact that $C_\sigma(X)=C_s(X)$.
\end{proof}

\subsection*{Acknowledgement.} The author wishes to thank the anonymous referee for his/her careful reading of the manuscript and his/her useful suggestions and comments.

\bibliographystyle{amsplain}

\end{document}